\documentclass[a4paper,12pt]{amsart}

\usepackage{amsmath}
\usepackage{amssymb}
\usepackage{ascmac}
\usepackage{amsthm}
\usepackage{color}
\usepackage{comment}
\usepackage{empheq}

\makeatletter

\@addtoreset{equation}{section}
\makeatother 

\theoremstyle{plain}
\newtheorem{thm}{Theorem}[section]
\newtheorem*{thm*}{Theorem}
\newtheorem{prop}[thm]{Proposition}
\newtheorem{lem}[thm]{Lemma}
\newtheorem{cor}[thm]{Corollary}

\theoremstyle{definition}

\theoremstyle{remark}
\newtheorem{rem}[thm]{Remark}

\newcommand{\Ric}{\operatorname{Ric}}
\renewcommand{\sec}{\operatorname{sec}}

\newcommand{\Hess}{\operatorname{Hess}}

\newcommand{\bm}{\partial M}

\newcommand{\ABP}{\operatorname{ABP}}
\newcommand{\CD}{\operatorname{CD}}
\newcommand{\RCD}{\operatorname{RCD}}
\newcommand{\CAT}{\operatorname{CAT}}
\newcommand{\MCP}{\operatorname{MCP}}

\newcommand{\eps}{\varepsilon}

\newcommand{\e}{\mathrm{e}}
\renewcommand{\d}{\mathrm{d}}

\usepackage{latexsym}
\def\qed{\hfill $\Box$} 

\makeatletter
\@namedef{subjclassname@2020}{%
  \textup{2020} Mathematics Subject Classification}
\makeatother

\title[Lower $0$-weighted Ricci curvature bound]{Geometric analysis on weighted manifolds\\ under lower $0$-weighted Ricci curvature bounds}

\author{Yasuaki Fujitani}
\address{Graduate School of Mathematical Sciences, The University of Tokyo, 3-8-1 Komaba, Tokyo, 153-8914, Japan}
\email{yasuakifujitani@g.ecc.u-tokyo.ac.jp}

\author{Yohei Sakurai}
\address{Department of Mathematics, Saitama University, 255 Shimo-Okubo, Sakura-ku, Saitama-City, Saitama, 338-8570, Japan}
\email{ysakurai@rimath.saitama-u.ac.jp}

\subjclass[2020]{Primary 53C21; Secondary 53C20}
\keywords{Weighted Ricci curvature; Eigenvalue estimate; ABP estimate; Sobolev inequality}
\date{\today}
\begin{document}
\maketitle

\begin{abstract}
We develop geometric analysis on weighted Riemannian manifolds under lower $0$-weighted Ricci curvature bounds.
Under such curvature bounds,
we prove a first non-zero Steklov eigenvalue estimate of Wang-Xia type on compact weighted manifolds with boundary,
and a first non-zero eigenvalue estimate of Choi-Wang type on closed weighted minimal hypersurfaces.
We also produce an ABP estimate and a Sobolev inequality of Brendle type.
\end{abstract}

\section{Introduction}

For $n\geq 2$,
let $(M,g,f)$ be an $n$-dimensional weighted Riemannian manifold,
namely,
$(M,g)$ is an $n$-dimensional complete Riemannian manifold, and $f\in C^{\infty}(M)$.
For $N\in (-\infty,+\infty]$,
the \textit{$N$-weighted Ricci curvature} is defined as follows (\cite{BE}, \cite{L}): 
\begin{equation*}
    \Ric_f^N:=\Ric + \Hess f-\frac{\d f\otimes \d f}{N-n}.
\end{equation*}
Here when $N = +\infty$, 
we interpret the last term of the right hand side as the limit $0$, 
and when $N = n$, 
we only consider a constant function $f$ and set $\Ric_f^n := \Ric$.
In this article,
we will focus on the case of $N=0$.
We study geometric and spectral properties of weighted manifolds under each of the following curvature bounds:
\begin{align}\label{eq:CD0}
    \Ric^0_f &\geq K\,g \text{~for $K\in \mathbb{R}$};\\ \label{eq:KL0}
    \Ric_f^0& \geq n\,\kappa\, \e^{-\frac{4f}{n}}g \text{~for $\kappa \in \mathbb{R}$}.
\end{align}

\subsection{Background on curvature conditions}

In recent years,
various curvature conditions have been introduced in the theory of weighted Riemannian manifolds.
It is well-known that
the classical curvature bound
\begin{equation}\label{eq:CD}
\Ric^{N}_{f} \geq K\,g
\end{equation}
for $N\in [n,+\infty]$ can be characterized by the \textit{curvature-dimension condition} $\CD(K,N)$ in the sense of Sturm \cite{S1}, \cite{S2} and Lott-Villani \cite{LV}.
Nowadays,
geometric and analytic properties under \eqref{eq:CD} are well-understood beyond the smooth setting due to the development of the theory of non-smooth metric measure spaces satisfying $\CD(K,N)$ or $\RCD(K,N)$ in the sense of Ambrosio-Gigli-Savar\'{e} \cite{AGS}.

The validity of the bound \eqref{eq:CD} for $N\in (-\infty,n)$ has been observed by Ohta \cite{O1}, Klartag \cite{K}, Kolesnikov-Milman \cite{KM1}, \cite{KM2}, Milman \cite{M}, Wylie \cite{W1} in view of the curvature-dimension condition, disintegration theory, functional inequalities, splitting theorems and so on.
By the following monotonicity of $\Ric^{N}_f$ with respect to $N$,
the bound \eqref{eq:CD} for $N\in (-\infty,n)$ is weaker than that for $N\in [n,+\infty]$:
For $N_1,N_2 \in [n,+\infty]$ with $N_1\leq N_2$,
and for $N_3,N_4 \in (-\infty,n)$ with $N_3\leq N_4$,
\begin{equation}\label{eq:monotonicity}
\Ric_f^{N_1}\leq  \Ric_f^{N_2} \leq \Ric_f^{\infty} \leq  \Ric_f^{N_3} \leq \Ric_f^{N_4}.
\end{equation}
Ohta \cite{O1}, \cite{O2} has extended the notion of $\CD(K,N)$ to $N\in (-\infty,0)$ in \cite{O1}, and to $N=0$ in \cite{O2}.
Recently, Magnabosco-Rigoni-Sosa \cite{MRS} and Oshima \cite{O} have investigated the stability of $\CD(K,N)$ for $N\in (-\infty,0)$,
and Magnabosco-Rigoni \cite{MR} have done optimal transport and local-to-global properties under $\CD(K,N)$ for $N\in (-\infty,0)$ in the non-smooth framework.

Wylie \cite{W1} has proven a splitting theorem of Cheeger-Gromoll type under \eqref{eq:CD} for $K=0$ and $N=1$,
where the bound \eqref{eq:CD} for $N=1$ is further weaker than that for $N\in (-\infty,0]$ by the monotonicity \eqref{eq:monotonicity}.
After that
Wylie-Yeroshkin \cite{WY} have introduced a variable bound 
\begin{equation}\label{eq:WY}
 \Ric^{1}_{f}\geq (n-1)\, \kappa\, \e^{-\frac{4f}{n-1}}\,g
\end{equation}
in view of the study of affine connections.
Notice that
this variable bound can be interpreted as a condition that the Ricci tensor for a certain affine connection is bounded from below by a constant $(n-1)\kappa$ along its geodesic,
where the connection is torsion-free and projectively equivalent to the Levi-Civita connection (more precisely, see \cite[Subsection 2.1]{WY}).
They have obtained diameter and volume comparisons,
and rigidity theorems by setting model spaces as warped or twisted product spaces.
Kuwae-Li \cite{KL} have provided a generalized bound
\begin{equation}\label{eq:KL}
\Ric_f^N\geq (n-N)\,\kappa\, \e^{-\frac{4f}{n-N}}g
\end{equation}
for $N\in (-\infty,1]$,
and also generalized the comparison geometry.
Furthermore,
Lu-Minguzzi-Ohta \cite{LMO} have introduced a curvature bound with \textit{$\eps$-range} which interpolates \eqref{eq:CD} for $K=(N-1)\kappa$, \eqref{eq:WY} and \eqref{eq:KL},
and examined geometric and analytic properties in a unified way.

\subsection{Main theorems}

The aim of this paper is to develop geometric analysis on weighted manifolds under lower $0$-weighted Ricci curvature bounds;
especially,
each of the curvature bounds \eqref{eq:CD0} and \eqref{eq:KL0}.
They are nothing but the bounds \eqref{eq:CD} and \eqref{eq:KL} for $N=0$,
respectively.
We derive eigenvalue estimates and functional inequalities concerning the \textit{weighted Laplacian} and \textit{weighted measure}
\begin{equation*}
\Delta_f := \Delta - \langle \nabla f, \nabla \cdot \rangle,\quad m_f := \e^{-f} v_g.
\end{equation*}
Our main theorems can be summarized as follows:
\begin{enumerate}\setlength{\itemsep}{4pt}
\item Wang-Xia type first non-zero Steklov eigenvalue estimate on compact weighted manifolds with boundary under \eqref{eq:CD0} for $K=0$ (see Theorem \ref{thm:Wang-Xia});
\item Choi-Wang type first non-zero eigenvalue estimate on closed weighted minimal hypersurfaces in weighted manifolds under \eqref{eq:CD0} for $K>0$ (see Theorem \ref{thm:Choi-Wang});
\item Alexandroff-Bakelman-Pucci type estimate (ABP estimate for short) on weighted manifolds under \eqref{eq:KL0} (see Theorem \ref{thm:ABP});
\item Brendle type Sobolev inequality under \eqref{eq:KL0} for $\kappa=0$ (see Theorem \ref{thm:Sobolev}).
\end{enumerate}

We first study the following Steklov eigenvalue problem on compact weighted manifolds with boundary:
\begin{align}\label{eq:steklov-boundary-problem}
    \begin{cases}
        {\Delta}_f u=0 & \text { on } M, \\ 
        u_\nu  = \lambda\, u & \text { on } \partial M,
    \end{cases}
\end{align}
where $\nu$ is the unit outer normal vector field on the boundary,
and $u_\nu$ is the derivative of $u$ in the direction of $\nu$.
In the unweighted case,
Wang-Xia \cite{WX} have obtained a first non-zero Steklov eigenvalue estimate on manifolds of non-negative Ricci curvature based on a first non-zero eigenvalue estimate of the weighted Laplacian on the boundary by Xia \cite{X} (see \cite[Theorem 1.1]{WX}, \cite[Theorem 1]{X}).
Kolesnikov-Milman \cite{KM2} have shown a Xia type estimate on weighted manifolds under \eqref{eq:CD0} for $K=0$ (see \cite[Theorem 1.3]{KM2}).
Also,
Batista-Santos \cite{BS} have proved a Wang-Xia type estimate under \eqref{eq:CD} for $K=0$ and $N\in [n,+\infty]$ (see \cite[Theorems 1.1 and 1.2]{BS}).
Based on \cite{KM2},
we extend the estimate in \cite{BS} to a weak setting \eqref{eq:CD0} for $K=0$.

We next examine the eigenvalue problem on closed weighted minimal hypersurfaces in weighted manifolds.
Here,
a weighted minimal hypersurface is a critical point of the weighted volume functional,
which is a crucial object in the study of self-similar solutions to the mean curvature flow.
In the unweighted case,
Choi-Wang \cite{CW} have obtained a lower bound of the first non-zero eigenvalue on closed minimal hypersurfaces in manifolds of positive Ricci curvature (see \cite[Theorem 2]{CW}, and also \cite{CS}).
Cheng-Mejia-Zhou \cite{CMZ}, Li-Wei \cite{LW}, Ma-Du \cite{MD} have derived a Choi-Wang type estimate on closed weighted minimal hypersurfaces in weighted manifolds under \eqref{eq:CD} for $K>0$ and $N=+\infty$ (see \cite[Theorem 2]{CMZ}, \cite[Theorem 7]{LW}, \cite[Theorem 3]{MD}). 
Note that
the topological constraint for ambient spaces is most relaxed in \cite{CMZ}.
We extend their estimate to a weak setting \eqref{eq:CD0} for $K>0$.

We further investigate the ABP estimate.
The ABP estimate plays a key role in the proof of the Krylov-Safanov Harnack inequality and the regularity theory for fully non-linear elliptic equations.
Cabr\'{e} \cite{C} has formulated an ABP estimate on Riemannian manifolds,
and Wang-Zhang \cite{WZ} have obtained such an estimate under a lower Ricci curvature bound (see \cite[Theorem 1.2]{WZ}).
It has been generalized for weighted manifolds under \eqref{eq:CD} with $N \in [n,+\infty]$ (see \cite[Proposition 3.2]{WZ}, \cite[Theorem 1.2]{WZ2}, \cite[Theorem 4.1]{MS}, and also \cite{H} in the non-smooth framework).
We produce an ABP estimate under a weak setting \eqref{eq:KL0}.

We finally consider the Sobolev inequality.
In the unweighted case,
Brendle \cite{B} has established a sharp Sobolev inequality on manifolds of non-negative Ricci curvature (see \cite[Theorem 1.1]{B}). 
Johne \cite{J} has obtained a Brendle type Sobolev inequality on weighted manifolds under \eqref{eq:CD} for $K=0$ and $N\in [n,+\infty)$ (see \cite[Theorem 1.1]{J}).
We will extend his estimate to a weak setting \eqref{eq:KL0} for $\kappa=0$.

The proof of eigenvalue estimates under \eqref{eq:CD0} is based on Bochner and Reilly type formulas associated with the $0$-weighted Ricci curvature.
On the other hand,
that of functional inequalities under \eqref{eq:KL0} is based on Riccati type inequalities.
Recall that
the curvature bound \eqref{eq:WY} concerning the $1$-weighted Ricci curvature was compatible with comparison geometry since the main study object was the distance function, whose Hessian degenerates in the direction of its gradient (cf. \cite[Lemma 3.1]{W2}).
Our main results can be available under lower $0$-weighted Ricci curvature bounds from the viewpoint that we analyze arbitrary functions which do not necessarily satisfy such a degenerate property.

\subsection{Organization}
In Section \ref{sec:Steklov eigenvalue estimate},
we provide a Reilly type formula associated with the $0$-weighted Ricci curvature (Proposition \ref{thm:0-reilly}),
improve the Xia type estimate in \cite{KM2} under the non-negativity of the $1$-weighted Ricci curvature (Theorem \ref{thm:KM}),
and conclude a Wang-Xia type first non-zero Steklov eigenvalue estimate (Theorem \ref{thm:Wang-Xia}).
In Section \ref{sec:Choi-Wang},
we investigate the Frankel property for weighted minimal hypersurfaces under the positivity of the $1$-weighted Ricci curvature (Proposition \ref{thm:1-Frankel-strong}),
and prove a Choi-Wang type first non-zero eigenvalue estimate (Theorem \ref{thm:Choi-Wang}).
In Section \ref{sec:ABP},
we establish an ABP estimate (Theorem \ref{thm:ABP}).
As applications of the ABP estimate,
we present a Krylov-Safonov type Harnack inequality (Theorem \ref{thm:KS}),
and a Liouville theorem for weighted harmonic functions (Corollary \ref{cor:Liouville}).
In Section \ref{sec:Sobolev},
we produce a Brendle type Sobolev inequality (Theorem \ref{thm:Sobolev}),
and derive an isoperimetric inequality (Corollary \ref{cor:isoperimetric}).

\section{Steklov eigenvalue estimate}\label{sec:Steklov eigenvalue estimate}
In this section,
we study the Steklov eigenvalue problem \eqref{eq:steklov-boundary-problem} on a compact weighted Riemannian manifold with boundary $(M,g,f)$.
We denote by $\lambda_{1,M}^{\mathrm{Ste}}$ the first non-zero Steklov eigenvalue for \eqref{eq:steklov-boundary-problem}.
On $\bm$,
let $\nu$ stand for the unit outer normal vector field.
The second fundamental form $\mathrm{II}_{\partial M}$,
the mean curvature $H_{\partial M}$ and the weighted mean curvature $H_{f,\partial M}$ are defined as 
\begin{equation}\label{def:weighted-mean-curvature}
    \mathrm{II}_{\partial M}(X,Y) := \langle \nabla_X \nu,Y \rangle, \quad H_{\partial M} := \mathrm{tr}\, \mathrm{II}_{\partial M},\quad H_{f,\partial M} := H_{\partial M} - f_\nu.
\end{equation}
Our first main result is the following Wang-Xia type first non-zero Steklov eigenvalue estimate:
\begin{thm}\label{thm:Wang-Xia}
Let $(M,g,f)$ be a compact weighted Riemannian manifold with boundary.
For $\sigma, k > 0$,
we assume $\Ric_f^0 \geq 0$, $\mathrm{II}_{\partial M} \geq \sigma \,g_{\partial M}$ and $H_{f,\partial M} \geq  k$.
Then
\begin{equation*}
\lambda_{1,M}^{\mathrm{Ste}} \leq \frac{\sqrt{\lambda_{1,\partial M}}\left(\sqrt{\lambda_{1,\partial M}}+\sqrt{\lambda_{1,\partial M}- k \sigma }\right)}{ k},
\end{equation*}
where $\lambda_{1,\partial M}$ is the first non-zero eigenvalue of the weighted Laplacian $\Delta_{f,\partial M}$ on $\bm$.
\end{thm}

\begin{rem}\label{rem:connXia}
Under the same setting as in Theorem \ref{thm:Wang-Xia},
the boundary $\bm$ is connected.
Moreover,
a Xia type estimate $\lambda_{1,\partial M}\geq k \sigma$ holds (see Theorems \ref{thm:Conn-Boundary} and \ref{thm:KM} below).
\end{rem}

\subsection{Reilly formula}\label{subsec:Reilly}
We derive a Reilly type formula associated with $\Ric^0_f$,
which is a key ingredient for the proof of Theorem \ref{thm:Wang-Xia}.
We begin with recalling the following Bochner type identity formulated by Wylie \cite{W2} (see \cite[Lemma 3.1]{W2}, and also \cite[Remark 4.4]{KLLS}):
\begin{prop}[\cite{W2}]\label{thm:0-bochner}
Let $(M,g,f)$ be an $n$-dimensional weighted Riemannian manifold.
Then for every $\varphi \in C^{\infty}(M)$,
\begin{align*}
    \frac{1}{2} \Delta_f |\nabla \varphi |^2=\Ric_f^0(\nabla \varphi, \nabla \varphi)   + \frac{\left(\Delta_f \varphi\right)^2}{n}+\left| \operatorname{Hess} \varphi-\frac{\Delta \varphi}{n} g\right|^2 + \e^{-\frac{2 f}{n}} \left\langle\nabla\left(\e^{\frac{2 f}{n}} \Delta_f \varphi\right), \nabla \varphi\right\rangle.
\end{align*}
\end{prop}

From Proposition \ref{thm:0-bochner},
we deduce the following Reilly type formula (cf. \cite[Theorem 1]{MD}, \cite[Theorem 1.1]{KM1}):
\begin{prop}\label{thm:0-reilly}
Let $(M,g,f)$ be an $n$-dimensional compact weighted Riemannian manifold with boundary.
Then for every $\varphi \in C^{\infty}(M)$,
\begin{align*}
    &\quad \,\,  \int_M\left\{ {\Delta}_f \varphi\left({\Delta}_{\left(1+\frac{2}{n}\right) f} \varphi - \frac{ {\Delta}_f \varphi}{n}\right)-\Ric_f^0({\nabla} \varphi, {\nabla} \varphi)-\left|\operatorname{Hess} \varphi-\frac{{\Delta} \varphi}{n} g\right|^2 \right\} \ \d m_f \\
    & =\int_{\partial M} \left(\varphi_\nu^{\,2} H_{f,\partial M} +  2 \varphi_\nu \Delta_{f, \partial M}\,  \psi  + \mathrm{II}_{\partial M}\left(\nabla_{\partial M} \,\psi, \nabla_{\partial M}\, \psi \right)  \right) \ \d m_{f,\partial M} ,\nonumber
\end{align*}
where $\psi := \varphi|_{\partial M}$.
\end{prop}
\begin{proof}
By integration by parts,
\begin{align*}
    &\quad \,\,  \int_M \e^{-\frac{2 f}{n}} \left\langle{\nabla}\left(\e^{\frac{2 f}{n}} {\Delta}_f \varphi\right), {\nabla} \varphi\right\rangle\  \d m_f\\
    &  =\int_M \left\langle {\nabla}\left(\e^{\frac{2 f}{n}} {\Delta}_f \varphi\right), {\nabla} \varphi\right\rangle \ \d m_{\left(1+\frac{2}{n}\right) f}\\
    & = \int_{\partial M}\left(\e^{\frac{2 f}{n}} \Delta_f \varphi \right) \varphi_\nu \ \d m_{\left(1+\frac{2}{n}\right) f, \partial M} -\int_M\left(\e^{\frac{2 f}{n}} {\Delta}_f \varphi\right) {\Delta}_{\left(1+\frac{2}{n}\right) f} \varphi \ \d m_{\left(1+\frac{2}{n}\right) f}\\
    & =\int_{\partial M}  \varphi_\nu \Delta_f \varphi  \ \d m_{f,\partial M} -\int_M {\Delta}_f \varphi {\Delta}_{\left(1+\frac{2}{n}\right) f} \varphi\ \d m_f.
\end{align*}
By using this equation and Proposition \ref{thm:0-bochner},
we see
\begin{align}
    &\quad \,\,\int_M \left\{{\Delta}_f \varphi\left({\Delta}_{\left(1+\frac{2}{n}\right) f} \varphi - \frac{ {\Delta}_f \varphi}{n}\right)-\Ric_f^0({\nabla} \varphi, {\nabla} \varphi)-\left|\operatorname{Hess} \varphi-\frac{{\Delta} \varphi}{n} g\right|^2 \right\}\ \d m_f\label{eq:0-reilly-2} \\
    &= \int_{\partial M} \varphi_\nu \Delta_f \varphi \ \d m_{f,\partial M} - \frac{1}{2} \int_M {\Delta}_f |{\nabla} \varphi|^2 \ \d m_f.\nonumber
\end{align}
The second term of the right-hand side of \eqref{eq:0-reilly-2} can be calculated as
\begin{align*}
    \frac{1}{2} \int_M \Delta_f|\nabla \varphi|^2 \ \d m_f &=\frac{1}{2} \int_{\partial M} \left(|\nabla \varphi|^2\right)_{\nu} \ \d m_{f, \partial M}\\
    &= \int_{\partial M} \varphi_\nu\left(\Delta_f \varphi-\Delta_{f, \partial M}\, \psi-\varphi_\nu H_{f,\partial M}\right) \ \d m_{f, \partial M} \\ 
    & \quad\,\, +\int_{\partial M}  \left(\langle \nabla_{\partial M}\, \varphi_\nu, \nabla_{\partial M}\, \psi\rangle - \mathrm{II}_{\partial M}(\nabla_{\partial M}\,\psi, \nabla_{\partial M}\,\psi)\right) \ \d m_{f, \partial M} \\ 
    & =\int_{\partial M} \varphi_\nu\left(\Delta_f \varphi-2 \Delta_{f, \partial M}\, \psi-\varphi_\nu H_{f,\partial M}\right) \ \d m_{f, \partial M} \\ 
    & \quad\,\, -\int_{\partial M} \mathrm{II}_{\partial M}(\nabla_{\partial M}\,\psi, \nabla_{\partial M}\,\psi)  \ \d m_{f, \partial M},
\end{align*}
where we used the argument in \cite[Theorem 1]{MD} (see also \cite[Chapter 8]{Li}).
This completes the proof.
\end{proof}

We also review a Reilly type formula for $\Ric^1_f$.
Li-Xia \cite{LX} have formulated a Reilly type formula for affine connections (see \cite[Theorem 1.1]{LX}).
As a special case,
we possess:
\begin{prop}[\cite{LX}]\label{thm:1-reilly}
Let $(M,g,f)$ be an $n$-dimensional compact weighted Riemannian manifold with boundary.
Then for every $\varphi \in C^{\infty}(M)$,
\begin{align*}
 &\quad \,\,\int_M  \left\{ \left(\Delta_{\frac{n}{n-1}f} \,\varphi\right)^2 - \Ric_f^1\left( \nabla \varphi, \nabla \varphi \right) - \left| \mathrm{Hess}\, \varphi - \frac{1}{n-1}\langle \nabla\varphi,\nabla f \rangle g\right|^2 \right\}\ \d m_f\\
    &= \int_{\partial M} \left(\varphi_\nu^{\,2} H_{f,\partial M} +  2 \varphi_\nu \Delta_{f, \partial M}\,  \psi  + \mathrm{II}_{\partial M}\left(\nabla_{\partial M} \,\psi, \nabla_{\partial M}\, \psi \right)  \right) \ \d m_{f,\partial M},
\end{align*}
where $\psi := \varphi|_{\partial M}$.
\end{prop}

\begin{rem}
By direct calculations,
for every $\varphi \in C^{\infty}(M)$ we see
\begin{align}\label{eq:calcution}
&\quad \,\,\left(\Delta_{f} \,\varphi\right)^2 - \Ric_f^\infty\left( \nabla \varphi, \nabla \varphi \right) - \left| \mathrm{Hess}\, \varphi\right|^2\\ \notag
&=\left(\Delta_{\frac{n}{n-1}f} \,\varphi\right)^2 - \Ric_f^1\left( \nabla \varphi, \nabla \varphi \right) - \left| \mathrm{Hess}\, \varphi - \frac{1}{n-1}\langle \nabla\varphi,\nabla f \rangle g\right|^2.
\end{align}
We can verify Proposition \ref{thm:1-reilly} by combining \eqref{eq:calcution} and the weighted Reilly formula associated with $\Ric^\infty_f$ in \cite{MD} (see \cite[Theorem 1]{MD}, and also \cite[Theorem 1.1]{KM1}).
\end{rem}

\subsection{Proof of Theorem \ref{thm:Wang-Xia}}\label{subsec:Wang-Xia}
The aim of this subsection is to prove Theorem \ref{thm:Wang-Xia}.
Before we do so,
we recall a connectedness principle for boundary proven by Wylie \cite{W2} (see \cite[Corollary 5.2]{W2}).
In our setting,
it can be stated as follows (cf. \cite[Theorem 4.1]{KM2}):
\begin{thm}[\cite{W2}]\label{thm:Conn-Boundary}
Let $(M,g,f)$ be a compact weighted Riemannian manifold with boundary.
If $\Ric_f^1 \geq 0$ and $H_{f,\partial M} > 0$,
then $\partial M$ is connected.
\end{thm} 

We also present a Xia type estimate.
Kolesnikov-Milman \cite{KM2} have proved such an estimate under the non-negativity of $\Ric^0_f$ (see \cite[Theorem 1.3]{KM2}).
We here point out that
it can be generalized under the non-negativity of $\Ric^1_f$ via Proposition \ref{thm:1-reilly}.
\begin{thm}\label{thm:KM}
Let $(M,g,f)$ be a compact weighted Riemannian manifold with boundary.
For $\sigma, k > 0$, we assume $\Ric_f^1 \geq 0$, $\mathrm{II}_{\partial M} \geq \sigma\, g_{\partial M}$ and $H_{f,\partial M} \geq  k$.
Then we have
\begin{equation*}
   \lambda_{1,\partial M} \geq k\sigma.
\end{equation*}
\end{thm}
\begin{proof}
Let $z$ be an eigenfunction for $\lambda_{1,\partial M}$,
and let $u$ be the solution to the following boundary value problem:
\begin{align*}
    \begin{cases}
      \Delta_{\frac{n}{n-1}f}\,u = 0 &\mbox{ on }M,\\
        u = z &\mbox{ on }\partial M.
    \end{cases}
\end{align*}
Using Proposition \ref{thm:1-reilly},
we obtain
\begin{align*}
    0 &\geq \int_{\partial M} \left( k u_\nu^2- 2 \lambda_{1,\partial M} \,u_\nu \, z + \sigma |\nabla_{\partial M}\,z|^2 \right)\ \d m_{f,\partial M}\\
    &= \int_{\partial M} \left( k u_\nu^2- 2\lambda_{1,\partial M}\, u_\nu\, z + \sigma\lambda_{1,\partial M}\, z^2 \right)\ \d m_{f,\partial M}\\
    &= \int_{\partial M} \left\{ k \left( u_\nu - \frac{\lambda_{1,\partial M}\, z}{k} \right)^2 + \lambda_{1,\partial M}\,z^2 \left( \sigma - \frac{\lambda_{1,\partial M}}{k} \right) \right\}  \d m_{f,\partial M}.\\
\end{align*}
It follows that
\begin{equation*}
    \sigma - \frac{\lambda_{1,\partial M}}{k}\leq 0.
\end{equation*}
This proves the desired estimate.
\end{proof}

We are now in a position to prove Theorem \ref{thm:Wang-Xia}.
\begin{proof}[Proof of Theorem \ref{thm:Wang-Xia}]
Let $z$ be an eigenfunction for $\lambda_{1,\partial M}$, and let $u$ be the solution to the following boundary value problem:
\begin{equation*}
    \begin{cases}
        \Delta_f u = 0 &\mbox{ on }M,\\
        u = z &\mbox{ on }\partial M.
    \end{cases}
\end{equation*}
By integration by parts,
and the Cauchy-Schwarz inequality,
\begin{align}\label{eq:Wang-Xia-1}
    \frac{\int_M|{\nabla} u|^2 \ \d m_f}{\int_{\partial M} z^2 \ \d m_{f,\partial M}} 
    &=\frac{ \int_{\partial M} u_\nu z \ \d m_{f,\partial M}}{\int_{\partial M} z^2 \ \d m_{f,\partial M}}\\ \notag
    & \leq \frac{\left(\int_{\partial M} z^2 \ \d m_{f,\partial M}\right)^{\frac{1}{2}}\left(\int_{\partial M}u_\nu^{\,2} \ \d m_{f,\partial M}\right)^{\frac{1}{2}}}{ \int_{\partial M} z^2 \ \d m_{f,\partial M}}=\frac{\left(\int_{\partial M} u_\nu^{\,2} \ \d m_{f,\partial M}\right)^{\frac{1}{2}}}{\left(\int_{\partial  M}z^2\ \d m_{f,\partial M}\right)^{\frac{1}{2}}}.
\end{align}
The first non-zero Steklov eigenvalue has the following variational characterization (see e.g., \cite{BS}):
\begin{equation}\label{eq:Wang-Xia-2}
   \lambda_{1,M}^{\mathrm{Ste}}=\inf_{w\in C^{\infty}(M)\backslash \{0\}} \left\{\frac{\int_M|{\nabla} w|^2\ \d m_f}{\int_{\partial M} w^2 \ \d m_{f,\partial M}}\ \bigg| \ \int_{\partial M} w \ \d m_{f,\partial M}=0\right\}.
\end{equation}
Since $\int_{\partial M} z \ \d m_{f,\partial M} = 0$, it follows from \eqref{eq:Wang-Xia-1} and \eqref{eq:Wang-Xia-2} that
\begin{equation}\label{eq:Wang-Xia-3}
   \lambda_{1,M}^{\mathrm{Ste}} \leq \frac{\left(\int_{\partial M} u_\nu^{\,2} \ \d m_{f,\partial M}\right)^{\frac{1}{2}}}{\left(\int_{\partial  M}z^2\ \d m_{f,\partial M}\right)^{\frac{1}{2}}}.
\end{equation}
Let us estimate the right-hand side of \eqref{eq:Wang-Xia-3}.
Due to Proposition \ref{thm:0-reilly}, we have 
\begin{align}
    0 &\geq \int_{\partial M}\left(u_\nu^{\,2} H_{f,\partial M} +2 u_\nu \Delta_{f,\partial M}\, z+\mathrm{II}_{\partial M}(\nabla_{\partial M}\, z, \nabla_{\partial M}\, z)\right) \ \d m_{f,\partial M} \label{eq:Wang-Xia-4}\\
    &\geq  k \int_{\partial M} u_\nu^{\,2} \ \d m_{f,\partial M} -2 \lambda_{1,\partial M} \int_{\partial M} u_\nu z \ \d m_{f,\partial M} + \sigma \lambda_{1,\partial M} \int_{\partial M} z^2 \ \d m_{f,\partial M} \nonumber\\
    & \geq k \int_{\partial M} u_\nu^{\,2} \ \d m_{f,\partial M} -2 \lambda_{1,\partial M}\left(\int_{\partial M} u_\nu^{\,2} \ \d m_{f,\partial M}\right)^{\frac{1}{2}}\left(\int_{\partial M} z^2 \ \d m_{f,\partial M}\right)^{\frac{1}{2}}\nonumber\\
    &\quad\,\, +\sigma \lambda_{1,\partial M} \int_{\partial M} z^2 \ \d m_{f,\partial M}\nonumber\\
    & =\frac{k \sigma -\lambda_{1,\partial M}}{\sigma} \int_{\partial M} u_\nu^{\,2} \ \d m_{f,\partial M}\nonumber\\
    &\quad\,\, +\left\{\sqrt{\frac{\lambda_{1,\partial M}}{\sigma}}\left(\int_{\partial M} u_\nu^{\,2} \ \d m_{f,\partial M}\right)^{\frac{1}{2}}-\sqrt{\sigma \lambda_{1,\partial M}}\left(\int_{\partial M} z^2 \ \d m_{f,\partial M}\right)^{\frac{1}{2}}\right\}^2\nonumber.
\end{align}
This leads us to
\begin{align*}
    &\quad\,\,\sqrt{\frac{\lambda_{1,\partial M}}{\sigma}}\left(\int_{\partial M} u_\nu^{\,2} \ \d m_{f,\partial M}\right)^{\frac{1}{2}}-\sqrt{\sigma \lambda_{1,\partial M}}\left(\int_{\partial M} z^2 \ \d m_{f,\partial M}\right)^{\frac{1}{2}}\\
   &\leq \frac{\sqrt{\lambda_{1,\partial M}-k \sigma}}{\sqrt{\sigma} }\left(\int_{\partial M} u_\nu^{\,2} \ \d m_{f,\partial M} \right)^{\frac{1}{2}};
\end{align*}
in particular,
\begin{align*}
   \frac{\sqrt{\lambda_{1,\partial M}}-\sqrt{\lambda_{1,\partial M}-k\sigma}}{\sqrt{\sigma}}\left(\int_{\partial M} u_\nu^{\,2} \ \d m_{f,\partial M}\right)^{\frac{1}{2}} \leq \sqrt{\sigma \lambda_{1,\partial M}}\left(\int_{\partial M} z^2 \ \d m_{f,\partial M}\right)^{\frac{1}{2}}.
\end{align*}
We conclude
\begin{align*}
   \left(\int_{\partial M} u_\nu^{\,2} \ \d m_{f,\partial M}\right)^{\frac{1}{2}} 
   &\leq \frac{\sqrt{\lambda_{1,\partial M}}\left(\sqrt{\lambda_{1,\partial M}}+\sqrt{\lambda_{1,\partial M}-k \sigma }\right)}{k}\left(\int_{\partial M} z^2 \ \d m_{f,\partial M} \right)^{\frac{1}{2}}.
\end{align*}
Combining this with \eqref{eq:Wang-Xia-3},
we complete the proof.
\end{proof}

\begin{rem}\label{rem:Steklov-rigidity}
We assume that the equality in Theorem \ref{thm:Wang-Xia} holds.
Then the equalities in \eqref{eq:Wang-Xia-4} also hold,
and hence 
\begin{align}\label{eq:equality-case}
    \Ric_f^0 (\nabla u,\nabla u) = 0,\quad  \Hess u = \frac{\Delta u}{n}g
\end{align}
on $M$, and 
\begin{equation*}
    H_{f,\partial M} \equiv k, \quad \mathrm{II}_{\partial M}(\nabla_{\partial M} \, z, \nabla_{\partial M} \, z) = \sigma |\nabla_{\partial M}\, z|^2, \quad u_\nu = \frac{\sqrt{\lambda_{1,\partial M}}\left(\sqrt{\lambda_{1,\partial M}} + \sqrt{\lambda_{1,\partial M} - k\sigma}\right)}{k} z
\end{equation*}
on $\partial M$.

We consider the case where the equality in Theorem \ref{thm:Wang-Xia} holds under a restrictive condition $\Ric_f^N \geq 0$ for $N \in (-\infty,0)$.
In this case,
by noticing 
\begin{align*}
 \Delta u = \langle \nabla u, \nabla f \rangle,\quad \Ric_f^0 (\nabla u,\nabla u) = \Ric_f^N (\nabla u,\nabla u) + \frac{N}{n(N-n)}(\Delta u)^2,
\end{align*}
we further see 
\begin{align*}
    \frac{N}{n(N-n)}(\Delta u)^2 = 0.
\end{align*}
Therefore,
$u$ is harmonic,
and the second identity in \eqref{eq:equality-case} leads to $\Hess u \equiv 0$.
In particular,
the argument of the proof of the rigidity statement in \cite[Theorem 1.1]{BS} works,
and hence the equality in Theorem \ref{thm:KM} also holds.

If the equality holds under a more restrictive condition $\Ric_f^N \geq 0$ for $N \in [n,+\infty)$,
then further rigidity properties can be available (cf. \cite[Theorem 1.6]{HR}).
\end{rem}

\begin{rem}
The authors do not know whether Theorem \ref{thm:Wang-Xia} can be extended to a weak setting as in Theorem \ref{thm:KM}.
It seems that Proposition \ref{thm:1-reilly} is not compatible with the integration by parts demonstrated in \eqref{eq:Wang-Xia-1}.
This remark can also be applied to Theorem \ref{thm:Choi-Wang} in the next section (see the calculation in \eqref{eq:Choi-Wang-1} below).
\end{rem}

\section{Eigenvalue estimate on weighted minimal hypersurface}\label{sec:Choi-Wang}

On a weighted Riemannian manifold $(M,g,f)$,
a hypersurface $\Sigma$ in $M$ is called \textit{$f$-minimal} if the weighted mean curvature $H_{f,\Sigma}$ vanishes identically on $\Sigma$,
where $H_{f,\Sigma}$ is defined as \eqref{def:weighted-mean-curvature} for a unit normal vector field $\nu$ on $\Sigma$.
In this section,
we study the eigenvalue problem on a closed weighted minimal hypersurface $\Sigma$.
We denote by $\lambda_{1,\Sigma}$ the non-zero eigenvalue for the weighted Laplacian $\Delta_{f,\Sigma}$.

Our second main theorem is the following Choi-Wang type first non-zero eigenvalue estimate:
\begin{thm}\label{thm:Choi-Wang}
Let $(M,g,f)$ be a compact weighted Riemannian manifold. 
For $K > 0$, we assume $\Ric_f^0 \geq K g$.
Let $\Sigma$ be a closed embedded $f$-minimal hypersurface in $M$.
Then
\begin{equation*}
    \lambda_{1,\Sigma} \geq \frac{K}{2}.
\end{equation*}
\end{thm}

\subsection{Frankel property}
We examine the so-called Frankel property for weighted minimal hypersurfaces under the positivity of $\Ric^1_f$,
which is a key ingredient for the proof of Theorem \ref{thm:Choi-Wang}.
For its proof,
we use some results concerning the relation between $\Ric^1_f$ and the weighted connection.
Let $(M,g,f)$ be an $n$-dimensional weighted Riemannian manifold,
and let $\varphi \in C^{\infty}(M)$ be another density function.
The \textit{weighted connection} is defined as 
\begin{equation*}
 \nabla^{\varphi}_XY := \nabla_X Y -\d \varphi(X)Y - \d \varphi(Y)X,
\end{equation*}
which is torsion free, affine, and projectively equivalent to the Levi-Civita connection $\nabla$.
The associated \textit{curvature tensor} and \textit{Ricci tensor} are defined as
\begin{equation*}
 R^\varphi(X, Y) Z := \nabla^{\varphi}_X \nabla^{\varphi}_Y Z-\nabla^{\varphi}_Y \nabla^{\varphi}_X Z-\nabla^{\varphi}_{[X, Y]} Z,\quad \Ric^\varphi(Y,Z) := \sum_{i = 1}^n \langle R^\varphi(E_i,Y)Z,E_i \rangle,
\end{equation*}
where $\{E_i\}^n_{i=1}$ is an orthonormal frame.

They are calculated as follows (see e.g., \cite[Proposition 3.3]{WY}):
\begin{lem}[\cite{WY}]
\begin{align}\notag
R^\varphi(X, Y) Z &=R(X,Y)Z+\Hess \varphi(Y,Z)X-\Hess \varphi(X,Z)Y\\ \notag
&\quad +\d\varphi(Y)\d\varphi(Z)X-\d\varphi(X)\d\varphi(Z)Y,\\ \label{eq:connection Ricci}
 \Ric^{\varphi} &=  \Ric + (n-1)\Hess \varphi + (n-1)\d \varphi \otimes \d \varphi.
\end{align}
In particular,
if $X,Y$ are orthogonal,
then
\begin{equation}\label{eq:weighted sectional}
\langle R^\varphi(X, Y) Y,X \rangle=\langle R(X, Y) Y,X \rangle+\nabla^2 \varphi(Y,Y)|X|^2+\d \varphi(Y)^2|X|^2.
\end{equation}
\end{lem}

\begin{rem}\label{rem:1-Ricci}
By \eqref{eq:connection Ricci},
one can observe that
if $\varphi=(n-1)^{-1}f$,
then $\Ric^{\varphi}=\Ric^1_f$.
\end{rem}

\begin{rem}\label{rem:weighted sectional}
If $X,Y$ are orthonormal,
then \eqref{eq:weighted sectional} coincides with the weighted sectional curvature $\overline{\sec}_{\varphi}(Y,X)$ introduced by Wylie \cite{W1} (see also \cite{KW}, \cite{KWY}).
\end{rem}

\begin{rem}\label{rem:second-variation-formula}
For a geodesic $\gamma:[0,d]\rightarrow M$,
the \textit{index form} is defined as
\begin{equation*}
I(V, V) :=\int_0^d\left(|V'(t)|^2- \left\langle R(V, \gamma'(t)) \gamma'(t), V\right\rangle \right) \ \d t.
\end{equation*} 
In view of \eqref{eq:weighted sectional},
the index form can be written as follows (see e.g., \cite[Proposition 5.1]{W1}):
If $V$ is perpendicular to $\gamma'$ along $\gamma$,
then
\begin{align*}
    &\quad\,\,I(V, V) \\
    & =\int_0^d\left(\left|V'(t)-\langle \nabla \varphi, \gamma'(t)\rangle V(t)\right|^2- \left\langle R^{\varphi}(V,\gamma'(t))\gamma'(t), V \right\rangle\right) \d t+\left[|V(t)|^2 \langle \nabla \varphi, \gamma'(t)\rangle\right]_{t =0}^{t = d}.
\end{align*}
\end{rem}

Let us show the following Frankel property:
\begin{prop}\label{thm:1-Frankel-strong}
Let $(M,g,f)$ be a complete weighted Riemannian manifold.
We assume ${\Ric_{f}^{1}} > 0$.
Let $\Sigma_1$ and $\Sigma_2$ be two closed embedded $f$-minimal hypersurfaces in $M$.
Then $\Sigma_1$ and $\Sigma_2$ must intersect.
\end{prop}
\begin{proof}
We give a proof by contradiction. 
Suppose $\Sigma_1 \cap \Sigma_2 = \emptyset$.
Let $\gamma : [0,d]\rightarrow M$ be a unit speed minimal geodesic from $\Sigma_1$ to $\Sigma_2$. 
Notice that $\gamma$ intersects with $\Sigma_1$ and $\Sigma_2$ orthogonally in virtue of the first variation formula.
Let $\{e_i\}^n_{i=1}$ be an orthonormal basis of $T_{\gamma(0)} M$ such that $e_n = \gamma'(0)$.
For each $i=1,\dots,n-1$, 
let $E_i$ stand for the parallel vector field along $\gamma$ with $E_i(0) = e_i$.
We set $\varphi:=(n-1)^{-1}f$ and $V_i := \e^{\varphi \circ \gamma}\,E_i$. 
Furthermore, 
let $\overline{\gamma}_i : [0,d] \times (-\epsilon,\epsilon)\rightarrow M$ be a variation of $\gamma$ whose variational vector field is $V_i$.
From the second variation formula and Remarks \ref{rem:1-Ricci} and \ref{rem:second-variation-formula},
it follows that
\begin{align*}
    0&\leq \sum_{i = 1}^{n-1}\int_0^d\left(|V'_i(t)|^2- \left\langle R(V_i, \gamma'(t)) \gamma'(t), V_i\right\rangle \right) \ \d t\\
    &\quad \,\,+ \sum_{i = 1}^{n-1} \left( -\mathrm{II}_{\Sigma_2}(V_i(d),V_i(d))+ \mathrm{II}_{\Sigma_1}(V_i(0),V_i(0))\right)\\
    & = \sum_{i=1}^{n-1}  \int_0^d \left(\left|V_i'(t)-  \langle \nabla \varphi, \gamma'(t)\rangle V_i(t)\right|^2- \left\langle R^{\varphi}(V_i,\gamma'(t))\gamma'(t),V_i \right\rangle  \right)\ \d t\nonumber\\
    &\quad\,\, +  \sum_{i = 1}^{n-1}\left[ |V_i(t)|^2 \langle \nabla \varphi, \gamma'(t)\rangle\right]_{t = 0}^{t = d} + \left( -\e^{2\varphi(\gamma(d))}\,H_{\Sigma_2}(\gamma(d))+\e^{2\varphi(\gamma(0))}\,H_{\Sigma_1}(\gamma(0))    \right)\nonumber\\
    &= -\int_0^d  \e^{2\varphi(\gamma(t))} \Ric_f^1(\gamma'(t),\gamma'(t))\ \d t+ \left( -\e^{2\varphi(\gamma(d))}\,H_{f,\Sigma_2}(\gamma(d))+\e^{2\varphi(\gamma(0))}\,H_{f,\Sigma_1}(\gamma(0))    \right)\\
    &= -\int_0^d  \e^{2\varphi(\gamma(t))} \Ric_f^1(\gamma'(t),\gamma'(t))\ \d t.
\end{align*}
Here,
$\mathrm{II}_{\Sigma_1},H_{\Sigma_1}(\gamma(0)),H_{f, \Sigma_1}(\gamma(0))$ are the second fundamental form, mean curvature, $f$-mean curvature for $\Sigma_1$ at $\gamma(0)$ with respect to $\gamma'(0)$,
and $\mathrm{II}_{\Sigma_2},H_{\Sigma_2}(\gamma(0)),H_{f, \Sigma_2}(\gamma(0))$ are those for $\Sigma_2$ at $\gamma(d)$ with respect to $\gamma'(d)$.
Now,
this contradicts with the positivity of $\Ric^1_f$.
Therefore, we arrive at the desired assertion.
\end{proof}

\begin{rem}
Wei-Wylie \cite{WW} have obtained the Frankel property under a stronger condition $\Ric^{\infty}_f>0$ (see \cite[Theorem 7.4]{WW}, and also \cite[Lemma 7.5]{LW}).
On the other hand,
Kennerd-Wylie \cite{KW} have proved a similar result for totally geodesic submanifolds in weighted manifolds of positive weighted sectional curvature (see \cite[Theorem 7.1]{KW}).
\end{rem}

\begin{rem}
It is well-known that
the Frankel property can be also derived from the Reilly formula when the ambient space is compact (see e.g., \cite[Lemma 5]{LW}).
We can derive Proposition \ref{thm:1-Frankel-strong} from Proposition \ref{thm:1-reilly} under the compactness.
Indeed,
along the line of the argument of the proof of \cite[Lemma 5]{LW},
it suffices to apply Proposition \ref{thm:1-reilly} to a suitable $(n/(n-1))f$-harmonic function instead of applying the classical weighted Reilly formula (\cite[(7)]{LW}) to a suitable $f$-harmonic function.
\end{rem}

\subsection{Proof of Theorem \ref{thm:Choi-Wang}}\label{subsec:Choi-Wang}
In this subsection,
let us give a proof of Theorem \ref{thm:Choi-Wang}.
We first recall a diameter bound obtained by Wylie-Yeroshkin \cite{WY}.
For an $n$-dimensional weighted Riemannian manifold $(M,g,f)$,
we consider a conformally deformed metric
\begin{equation*}
g_{f}:=\e^{-\frac{4f}{n-1}}g.
\end{equation*}
We possess the following (see \cite[Theorem 2.2]{WY}):
\begin{prop}[\cite{WY}]\label{thm:wylie-fundamental-group}
Let $(M,g,f)$ be an $n$-dimensional complete weighted Riemannian manifold. 
For $\kappa > 0$, 
we assume $\Ric_f^1 \geq (n-1)\kappa \,\e^{-\frac{4f}{n-1}}g$.
Then the diameter of $M$ induced from the metric $g_f$ is bounded from above by $\pi/\sqrt{\kappa}$.
\end{prop}

We are now in a position to prove Theorem \ref{thm:Choi-Wang}.
\begin{proof}[Proof of Theorem \ref{thm:Choi-Wang}]
We first prove the assertion in the case where $M$ is simply connected.
In this case,
$M$ and $\Sigma$ are orientable since $\Sigma$ is a closed embedded hypersurface.
Due to Proposition \ref{thm:1-Frankel-strong} together with the argument of \cite[Lemma 6]{LW}, 
$\Sigma$ is connected,
and it divides $M$ into two components $U_1$ and $U_2$.
Let $z$ be an eigenfunction for $\lambda_{1,\Sigma}$.
We may assume
\begin{equation*}
    \int_{\partial U_1} \mathrm{II}_{\partial U_1}(\nabla_{\partial U_1}\, z,\nabla_{\partial U_1}\, z)\ \d m_{f,\partial U_1} \geq 0.
\end{equation*}
Let $u$ be the solution to the following boundary value problem:
\begin{equation*}\label{eq:Choi-Wang-boundary-problem-1}
    \begin{cases}
        \Delta_f u = 0 & \mbox{ on }U_1,\\
        u = z & \mbox{ on }\partial U_1.
    \end{cases}
\end{equation*}
Applying Proposition \ref{thm:0-reilly} to $u$ on $U_1$, we derive
\begin{align}\label{eq:Choi-Wang-1}
    0&\geq K \int_{U_1} |\nabla u|^2 \ \d m_f - 2\lambda_{1,\Sigma}\int_{\partial U_1} u_\nu \, z \ \d m_{f,\partial U_1} + \int_{\partial U_1} \mathrm{II}_{\partial U_1}(\nabla_{\partial U_1} \, z, \nabla_{\partial U_1} \, z)\ \d m_{f,\partial U_1}\\ \notag
    &\geq (K- 2\lambda_{1,\Sigma}) \int_{U_1} |\nabla u|^2 \ \d m_f,
\end{align}
where $\nu$ denotes the outer unit normal vector field on $\partial U_1$.
This yields
\begin{equation*}
    \lambda_{1,\Sigma} \geq \frac{K}{2}.
\end{equation*}
 
We next consider the case where $M$ is not simply connected. 
Note that
$M$ satisfies
\begin{equation}\label{eq:WY curvature}
    \Ric_f^1 \geq \Ric_f^0 \geq K \,g \geq (K\e^{\frac{4\,\inf f}{n-1}})\,\e^{-\frac{4f}{n-1}}\,g
\end{equation}
in view of the compactness of $M$;
in particular,
Proposition \ref{thm:wylie-fundamental-group} can be applied.
Let $\overline{M}$ stand for the universal covering with covering map $\pi:\overline{M}\rightarrow M$.
Set
\begin{equation*}
\bar{g}:= \pi^{\ast}g,\quad \bar{f} := \pi^{\ast}f,\quad \bar{g}_f:= \pi^{\ast}g_f=\e^{-\frac{4\bar{f}}{n-1}}\bar{g}.
\end{equation*}
The bound \eqref{eq:WY curvature} can be lifted over $(\overline{M},\bar{g},\bar{f})$,
and also $(\overline{M},\bar{g}_f)$ is complete by the compactness of $M$.
Therefore,
by Proposition \ref{thm:wylie-fundamental-group},
$\overline{M}$ is compact;
in particular,
the fundamental group of $M$ is finite.
Now,
the lift $\overline{\Sigma}$ of $\Sigma$ is a closed embedded $\bar{f}$-minimal hypersurface in $\overline{M}$.
The consequence for simply connected manifolds observed in the above paragraph tells us that
the first non-zero eigenvalue $\lambda_{1,\overline{\Sigma}}$ of $\Delta_{\overline{f},\overline{\Sigma}}$ satisfies
\begin{equation*}
\lambda_{1,\Sigma}\geq \lambda_{1,\overline{\Sigma}} \geq \frac{K}{2}.
\end{equation*}
Thus,
we complete the proof.
\end{proof}

\begin{rem}
Choi-Wang type estimate has been investigated in \cite{LW}, \cite{CMZ} under a stronger assumption $\Ric^{\infty}_f\geq Kg$ (see \cite[Theorem 7]{LW}, \cite[Theorem 2]{CMZ}).
The setting of Theorem \ref{thm:Choi-Wang} corresponds to that in \cite{LW} in the sense that the ambient space is assumed to be compact.
In \cite{CMZ},
the ambient space is assumed to be complete,
and two hypersurfaces are assumed to be contained in a bounded convex domain.
Modifying the above proof along the line of \cite{CMZ},
one can deduce the estimate in Theorem \ref{thm:Choi-Wang} under the same setting as in \cite{CMZ}.
\end{rem}

\section{ABP estimate}\label{sec:ABP}

This section is devoted to the study of the ABP estimate and its application on a complete weighted Riemannian manifold $(M,g,f)$.
The aim of the ABP estimate to estimate the size of \textit{contact sets}.
We recall the formulation of the contact set on manifolds suggested by Cabr\'{e} \cite{C},
which uses the square of distance functions as touching functions.
For a positive constant $a> 0$, 
a compact subset $E$, 
a bounded domain $U$ and $u\in C(\overline{U})$, 
the \textit{contact set} $A_a(E,U,u)$ is defined by the set of all $x\in \overline{U}$ such that
\begin{equation*}
 \inf_{\overline{U}}\left(u + \frac{a}{2} d_y^2\right) = u(x) + \frac{a}{2} d^2_y(x)
\end{equation*}
for some $y\in E$, where $d_y(x) := d(x,y)$.
For $\kappa \in \mathbb{R}$, 
we denote by $s_{\kappa}(t)$ a unique solution to the Jacobi equation $\psi''(t)+\kappa\,\psi(t)=0$ with $\psi(0)=0$ and $\psi'(0)=1$.
We set
\begin{equation*}
\mathcal{S}_\kappa (t):=\frac{s_{\kappa}(t)}{t},\quad \mathcal{H}_\kappa(t):=t\,\frac{s'_{\kappa}(t)}{s_{\kappa}(t)}
\end{equation*}
such that $\mathcal{S}_\kappa (0)=1$ and $\mathcal{H}_\kappa (0)=1$.

The third main result is the following ABP estimate:
\begin{thm}\label{thm:ABP}
Let $(M,g,f)$ be an $n$-dimensional  complete weighted Riemannian manifold. 
For $\kappa \in \mathbb{R}$, 
we assume $\Ric_f^0   \geq  n\,\kappa \,\e^{-\frac{4f}{n} } \,g$.
For a positive constant $a> 0$, 
a compact subset $E$, 
a bounded domain $U$ and $u\in C^2(\overline{U})$, 
we assume that $A_a(E,U ,u)$ is contained in $U$. 
Then
\begin{equation*}
    m_f(E) \leq \int_{A_a(E,U ,u)} \left\{\mathcal{S}_\kappa\left(\frac{1}{a}s_{a,f,u}   \,|\nabla u| \right) \left( \mathcal{H}_\kappa \left( \frac{1}{a}s_{a,f,u} \,|\nabla u|  \right) +\frac{1 }{na}\e^{\frac{2f}{n}}\,s_{a,f,u}\,\Delta_f u  \right)\right\}^n \ \d m_f, 
\end{equation*}
where $s_{a,f,u}:\overline{U}\to \mathbb{R}$ is a function defined by
\begin{equation*}
    s_{a,f,u}(x) :=  \int_0^1 \exp\left\{ -\frac{2}{n} f\left(\exp_x\left( \frac{\xi (\nabla u)_x}{a} \right)\right) \right\} \d\xi.
\end{equation*}
\end{thm}

\subsection{Proof of Theorem \ref{thm:ABP}}\label{subsec:ABP}

We prove the desired estimate along the line of the argument of the proof of \cite[Theorem 1.2]{WZ}.
The key step is to yield a Riccati type inequality associated with the $0$-weighted Ricci curvature,
which enables us to adopt the argument in \cite{WZ} to our setting (see \eqref{eq:ABP-2} below).

Set $A_a:=A_a(E,U ,u)$.
We define a map $\Phi_{a} : A_a \to M$ by
\begin{equation*}
\Phi_{a}(x) := \exp_x\left(\frac{(\nabla u)_x}{a}\right).
\end{equation*}
By the same argument as in the proof of \cite[Theorem 1.2]{WZ},
the map $\Phi_{a}$ is differentiable, and it maps $A_a$ onto $E$.
For $t\in [0,1]$, we also define a map $\Phi_{a}^t : A_a \to M$ by
\begin{equation*}
\Phi_{a}^t ( x):=\exp_x\left( \frac{t(\nabla u)_x}{a} \right).
\end{equation*}
Furthermore, for $t \in [0,1]$ and $x \in A_a$, we set 
\begin{equation*}
 l(t,x) := \log  (\det (d\Phi_{a}^t)_x),\quad l_f(t,x) := l(t,x)  - f(\Phi_{a}^t(x)) + f(x), \quad J_f(t,x) := \exp \left(\frac{l_f (t,x) }{n} \right),
\end{equation*}
which are well-defined (see \cite[Theorem 1.2]{WZ}).

We fix $x\in A_a$,
and write
\begin{equation*}
\gamma(t) := \Phi_{a}^t(x),\quad l(t) := l(t,x),\quad l_f(t) := l_f(t,x),\quad J_f(t) := J_f(t,x).
\end{equation*}
Note that it enjoys the following Riccati inequality (see e.g., \cite[Chapter 14]{V}, \cite[Theorem 1.2]{WZ}):
\begin{equation*}
    l''(t) \leq -\frac{l'(t)^2}{n} - \Ric(\gamma'(t),\gamma'(t)).
\end{equation*}
It follows that
\begin{equation}\label{eq:ABP-1}
    l_f''(t)\leq -\frac{l'(t)^2}{n} - \Ric_f^{\infty}\left(\gamma'(t),\gamma'(t)\right)
    = -\frac{l_f'(t)^2}{n} - \frac{2}{n} l_f'(t) \left\langle \nabla f , \gamma'(t) \right\rangle   - \Ric_f^{0}\left(\gamma'(t),\gamma'(t)\right).
\end{equation}
Let $s_f:[0,1]\to \mathbb{R}$ be a function defined by
\begin{equation*}
    s_{f}(t)  := \int_0^t \exp\left( -\frac{2f(\gamma(\xi))}{n} \right) \d \xi,
\end{equation*}
and let $t_f:[0,s_f(1)]\to [0,1]$ be its inverse function. 
For $\widehat{l}_f(s) := l_f(t_f(s))$,  
the estimate \eqref{eq:ABP-1} and the curvature assumption lead us to
\begin{equation}\label{eq:ABP-2}
    \widehat{l}_f''(s)\leq -\frac{\widehat{l}_f'(s)^2}{n}- \Ric_f^0\left(\gamma'(t_f(s)),\gamma'(t_f(s))\right)\exp\left( \frac{4f(\gamma(t_f(s)))}{n} \right)\leq - \frac{\widehat{l}_f'(s)^2}{n} - \frac{n \kappa}{a^2}  |(\nabla u)_x|^2;
\end{equation}
in particular,
for $\widehat{J}_f(s) := J_f(t_f(s))$,
we conclude
\begin{equation*}
    \widehat{J}_f''(s) \leq - \frac{\kappa}{a^2}  |(\nabla u)_x|^2 \widehat{J}_f(s),\quad \widehat{J}_f  (0) = 1,\quad \widehat{J}_f'  (0)=  \frac{\e^{\frac{2f(x)}{n}} \Delta_f u(x)}{na}.
\end{equation*}
We define
\begin{align*}
    J_\kappa(s) := \mathcal{S}_\kappa\left( \frac{ |(\nabla u)_x|}{a}     \,s \right) \left( \mathcal{H}_\kappa\left(\frac{|(\nabla u)_x| }{a}  \, s \right) +\frac{\e^{\frac{2f(x)}{n}} \Delta_f u (x) }{na} \, s \right).
\end{align*}
By straightforward calculations, we have
\begin{equation*}
 J_\kappa''(s) = -\frac{\kappa}{a^2} |(\nabla u)_x|^2 \, J_\kappa(s),\quad   \lim_{s\to 0} J_\kappa(s) = 1, \quad \lim_{s \to 0}J_\kappa'(s) = \frac{\e^{\frac{2f(x)}{n}} \Delta_f u(x)}{na}.
\end{equation*}
Applying the ODE comparison to $\widehat{J}_f$ and $J_\kappa$,
we see
\begin{equation*}
 J_f(1,x) = \widehat{J}_f(s_f(1)) \leq J_\kappa(s_f(1)).
\end{equation*}
This together with the change of variable, 
we arrive at
\begin{align*}
    m_f(E)
    &\leq \int_{A_a} \e^{-f(\Phi_{a}(x))} |\det (d \Phi_{a})_x)|\ \d v_g(x)\label{eq:ABP-4}\\
    &= \int_{A_a} J_f(1,x)^n \ \d m_f (x)\nonumber\\
    &\leq  \int_{A_a} \left\{\mathcal{S}_\kappa\left(\frac{1}{a}\,s_{a,f,u} \,|\nabla u|  \right) \left( \mathcal{H}_\kappa \left( \frac{ 1}{a}s_{a,f,u} |\nabla u| \right) +\frac{1 }{na}\,\e^{\frac{2f}{n}} \, s_{a,f,u}\, \Delta_f u \right)\right\}^n \d m_f\nonumber,
\end{align*}
where we used the surjectivity of $\Phi_{a}$ in the first inequality.
This completes the proof.
\qed

\subsection{Harnack inequality}
By Theorem \ref{thm:ABP},
we can obtain a Krylov-Safonov type Harnack inequality under the boundedness of the density function.
We will give an outline of its proof along a standard argument in previous works (cf. \cite[Theorem 2.1]{C}, \cite[Theorem 1.1]{Kim}, \cite[Theorems 1.4, 1.5]{WZ}, \cite[Theorems 1.3, 1.4 and 1.5]{WZ2}).

We first recall the following Laplacian and volume comparison estimates established by Kuwae-Li \cite{KL} (see \cite[Theorem 2.2 and Lemma 6.1]{KL}):
\begin{prop}[\cite{KL}]\label{thm:Laplacian}
Let $(M, g,f)$ be an $n$-dimensional complete weighted Riemannian manifold. 
For $\kappa \leq 0$ and $b \geq  0$, 
we assume $\Ric_f^0 \geq n \,\kappa\, \mathrm{e}^{-\frac{4f}{n}} \,g$ and $|f| \leq b$.
Let $o\in M$.
Then
\begin{align*}
\Delta_f d_o \leq \frac{n\, \e^{\frac{4b}{n}}}{d_o} \mathcal{H}_\kappa\left( \e^{-\frac{2b}{n}} \, d_o \right)
\end{align*}
outside $\{o\} \cup \mathrm{Cut}(o)$,
where $\mathrm{Cut}(o)$ denotes the cut locus of $o$.
\end{prop}

\begin{prop}[\cite{KL}]\label{thm:doubling}
Let $(M, g,f)$ be an $n$-dimensional complete weighted Riemannian manifold. 
For $\kappa \leq 0$ and $b \geq  0$, 
we assume $\Ric_f^0 \geq n \,\kappa\, \mathrm{e}^{-\frac{4f}{n}} \,g$ and $|f| \leq b$.
Let $o\in M$.
Then for all $r_1,r_2>0$ with $r_1<r_2$,
we have
\begin{equation*}
    \frac{m_f(B_{r_2}(o))}{m_f(B_{r_1}(o))} \leq \e^{4b}\left(\frac{r_2}{r_1}\right)^{n + 1}\exp\left( r_2\,\e^{\frac{2b}{n}}\sqrt{-n\,\kappa}\right).
\end{equation*}
\end{prop}

Proposition \ref{thm:Laplacian} enables us to construct the following barrier function (cf. \cite[Lemma 5.5]{C}, \cite[Lemma 3.3]{Kim}, \cite[Lemma 4.2]{WZ}, \cite[Lemma 6.4]{WZ2}):
\begin{lem}\label{lem:for-KS-2}
Let $(M, g,f)$ be an $n$-dimensional complete weighted Riemannian manifold. 
For $\kappa \leq 0$ and $b \geq  0$, 
assume $\Ric_f^0 \geq n \,\kappa\, \mathrm{e}^{-\frac{4f}{n}} \,g$ and $|f| \leq b$.
Let $\delta_1\in (0,1)$ and $r>0$.
Let $o\in M$.
Then there exists $\phi \in C(M)$ such that the following hold:
\begin{enumerate}\setlength{\itemsep}{+1.0mm}
\item $\phi \geq 0$ on $B_{r}(o)\setminus B_{3r/4}(o)$, and $\phi\leq -2 $ on $B_{r/2}(o)$;
\item $\phi\geq -C_1$ on $B_{r}(o)$;
\item we have 
    \begin{equation*}
        \mathcal{H}_\kappa\left(2r\,\e^{\frac{2b}{n}} \right) + \frac{r^2}{n} \e^{\frac{4b}{n}} \Delta_f \phi \leq 0
    \end{equation*}
    on $B_{r}(o) \setminus B_{\delta_1 r}(o)$ outside $\mathrm{Cut}(o)$;
\item we have 
    \begin{equation*}
            \mathcal{H}_\kappa\left( 2r\,\e^{\frac{2b}{n}} \right) + \frac{r^2}{n}\e^{\frac{4b}{n}}  \Delta_f \phi \leq C_2
    \end{equation*}
    on $B_{\delta_1 r}(o)$ outside $\mathrm{Cut}(o)$.
\end{enumerate}
Here, $C_1$ and $C_2$ are positive constants depending only on $n,b,\delta_1,r\sqrt{-\kappa}$.
\end{lem}

Furthermore,
based on Lemma \ref{lem:for-KS-2},
one can show the following  (cf. \cite[Lemma 4.3]{WZ}):
\begin{lem}\label{lem:for-KS-3}
Let $(M, g,f)$ be an $n$-dimensional complete weighted Riemannian manifold. 
For $\kappa \leq 0$ and $b \geq  0$, 
assume $\Ric_f^0 \geq n \,\kappa\, \mathrm{e}^{-\frac{4f}{n}} \,g$ and $|f| \leq b$.
Let $\delta_1\in (0,1)$ and $r>0$.
Let $o\in M$.
Let $\phi\in C(M)$ be a function obtained in Lemma $\ref{lem:for-KS-2}$,
and let $u\in C^{\infty}(M)$ be a function satisfying 
\begin{equation*}
    u \geq 0 \mbox{ on } B_{r}(o),\quad  \inf_{B_{r/2}(o)} u \leq 1.
\end{equation*}
Set $a := r^{-2}$ and $w := u + \phi$.
Then there exists $x_0\in B_{3r/4}(o)$ such that
\begin{equation*}
    A_a(\overline{B}_{r/2}(x_0), B_{r}(o), w) \subset B_{3r/4}(o) \cap \{u \leq C_3\},
\end{equation*}
where $C_3$ is a positive constant depending only on $n,b,\delta_1, r\sqrt{-\kappa}$.
\end{lem}

Now,
Theorem \ref{thm:ABP} together with Lemmas \ref{lem:for-KS-2} and \ref{lem:for-KS-3} leads us to the following key assertion (cf. \cite[Lemma 3.1]{Kim}):
\begin{prop}\label{thm:for-KS-1}
Let $(M, g,f)$ be an $n$-dimensional complete weighted Riemannian manifold. 
For $\kappa \leq 0$ and $b \geq  0$, 
assume $\Ric_f^0 \geq n \,\kappa\, \mathrm{e}^{-\frac{4f}{n}} \,g$ and $|f| \leq b$.
Let $\delta_1\in (0,1)$ and $r>0$.
Then there exist positive constants $\delta_2, \delta_3, C_4$ depending only on $n, b, \delta_1,r\sqrt{-\kappa}$ such that the following hold:
For $o\in M$ and $F\in C^{\infty}(M)$,
let $u\in C^{\infty}(M)$ satisfy
\begin{equation*}
    \Delta_f u \leq F \mbox{ on } B_{r}(o) ,\quad u \geq 0 \mbox{ on } B_{r}(o),\quad  \inf_{B_{r/2}(o)} u \leq 1.
\end{equation*}
If
\begin{equation*}
    r^2 \left( \frac{1}{m_f(B_{r}(o)) }\int_{B_{r}(o)} |F|^{n} \ \d m_f \right)^{\frac{1}{n}} \leq \delta_2,
\end{equation*}
then 
\begin{equation*}
    \frac{m_f(\{u\leq C_4\}\cap B_{\delta_1 r}(o))}{m_f(B_{r}(o)) } \geq \delta_3.
\end{equation*}
\end{prop}
\begin{proof}
Let $A_a :=  A_a(\overline{B}_{r/2}(x_0),B_{r}(o), w)$ be the contact set obtained in Lemma \ref{lem:for-KS-3}.
Notice that
we may assume $\mathrm{Cut}(o)\cap A_a = \emptyset$ (see \cite[Lemma 4.4]{WZ} and \cite[Lemmas 3.6 and 6.4]{WZ2}).
For every $x \in A_a$, 
there exists $y \in \overline{B}_{r/2}(x_0)$ such that
\begin{equation*}
    (\nabla w)_x + a d_y(x) (\nabla d_{y})_x = 0;
\end{equation*}
in particular,
Lemma \ref{lem:for-KS-3} implies
\begin{equation*}\label{eq:for-ks-thm-1}
a^{-1} |(\nabla w)_x| = d(x,y) \leq d(x,o) + d(o,x_0) + d(x_0,y) < 3r/4 + 3r/4 + r/2 = 2r.
\end{equation*}
Therefore,
due to Theorem \ref{thm:ABP},
\begin{equation*}
    m_f(B_{r/2}(x_0))\leq \int_{A_a}\left\{\mathcal{S}_\kappa\left(2r\,\e^{\frac{2b}{n}} \right)\left( \mathcal{H}_\kappa\left(2r\,\e^{\frac{2b}{n}}\right) + \frac{1}{na}  \e^{\frac{4b}{n}}\Delta_f w \right)\right\}^n\ \d m_f.
\end{equation*}
Furthermore,
Lemma \ref{lem:for-KS-2} yields
\begin{align*}
    m_f(B_{r/2}(x_0))^{\frac{1}{n}}
    &\leq \mathcal{S}_\kappa\left(2r\,\e^{\frac{2b}{n}} \right) \left\{\int_{A_a}\left( \frac{1}{na} \e^{\frac{4b}{n}}   |F| + C_2  \,  \chi_{A_a\cap  B_{\delta_1 r}(o)} \right)^n \ \d m_f\right\}^{\frac{1}{n}}\\
    &\leq \mathcal{S}_\kappa\left(2r\,\e^{\frac{2b}{n}} \right) \left\{ \frac{1}{na}\e^{\frac{4b}{n}}\left(\int_{B_{r}(o)} |F|^n\ \d m_f\right)^{\frac{1}{n}} + C_2 \,m_f(A_a\cap B_{\delta_1 r}(o))^{\frac{1}{n}} \right\},
\end{align*}
where $C_2$ is a positive constant in Lemma \ref{lem:for-KS-2}.
On the other hand,
by Proposition \ref{thm:doubling},
\begin{equation*}
    \frac{m_f(B_{r}(o))}{m_f(B_{r/2}(x_0)) } \leq \frac{m_f(B_{7r/4}(x_0)) }{m_f(B_{r/2}(x_0)) }\leq C_5
\end{equation*}
since $x_0\in B_{3r/4}(o)$,
where $C_5$ is a positive constant depending only on $n,b,r\sqrt{-\kappa}$.
Combining these estimates, we see 
\begin{equation*}
 C_5^{-\frac{1}{n}} \leq \mathcal{S}_\kappa\left(2r\,\e^{\frac{2b}{n}} \right) \left\{\frac{1}{na}\left(\frac{1}{m_f(B_{r}(o)) }\int_{B_{r}(o)}|F|^n\ \d m_f\right)^{\frac{1}{n}} + \frac{C_2 \,m_f(\{u\leq C_3\}\cap B_{\delta_1 r}(o))^{\frac{1}{n}}}{m_f(B_{r}(o))^{\frac{1}{n}}} \right\},
\end{equation*}
where $C_3$ is a positive constant obtained in Lemma \ref{lem:for-KS-3}.
Taking $\delta_2$ sufficiently small, we complete the proof.
\end{proof}

Once we obtain Proposition \ref{thm:for-KS-1},
the following desired Harnack inequality follows from a standard covering argument:
\begin{thm}\label{thm:KS}
Let $(M, g,f)$ be an $n$-dimensional complete weighted Riemannian manifold. 
For $\kappa \leq 0$ and $b \geq  0$, 
we assume $\Ric_f^0 \geq n \,\kappa\, \mathrm{e}^{-\frac{4f}{n}} \,g$ and $|f| \leq b$.
For $o\in M$ and $r > 0$, 
let $ u \in C^{\infty}(M)$ satisfy $u \geq 0$ on $B_{2r}(o)$.
Then we have
\begin{equation*}
    \sup_{B_r(o)} u \leq C\left\{ \inf_{B_r(o)} u  + r^2 \left( \frac{1}{m_f(B_{2r}(o))}\int_{B_{2r}(o)} |\Delta_f u|^{n}  \ \d m_f  \right)^{\frac{1}{n}}\right\},
\end{equation*}
where $C$ is a positive constant depending only on $n,b,r\sqrt{-\kappa}$.
\end{thm}

As an application of Theorem \ref{thm:KS}, 
we have the following Liouville theorem (cf. \cite[Corollary 2.2]{C}, \cite[Corollary 1.2]{Kim}):
\begin{cor}\label{cor:Liouville}
Let $(M,g,f)$ be a complete weighted Riemannian manifold.
We assume $\Ric_f^0\geq 0$ and $f$ is bounded.
Let $u \in C^{\infty}(M)$ be an $f$-harmonic function $($i.e., $\Delta_f u = 0$$)$ bounded from below.
Then $u$ must be constant.
\end{cor}

\begin{rem}
Under the same setting as in Corollary \ref{cor:Liouville},
the boundedness of $u$ can be replaced with a sublinear growth condition.
This assertion has been proven by the first named author \cite{F} under an assumption $\Ric_f^N\geq 0$ with $N \in (-\infty,0)$ (see \cite[Theorem 5.2]{F}).
One can extend it to a weak setting $\Ric_f^0\geq 0$ by observing that
every $f$-harmonic function $u$ satisfies $\Delta_f |\nabla u|^2\geq 0$ by virtue of Proposition \ref{thm:0-bochner},
and by using a mean value inequality (see \cite[Theorem 5.1]{F}, and also \cite[Corollary 4.3]{KLLS}).
\end{rem}

\section{Sobolev inequality}\label{sec:Sobolev}

The last one is the following Sobolev inequality of Brendle type:
\begin{thm}\label{thm:Sobolev}
Let $(M,g,f)$ be an $n$-dimensional complete weighted Riemannian manifold, and let $\Omega$ be a bounded domain in $M$ with smooth boundary. 
We assume $\Ric_f^0 \geq 0$ and $f \geq b$ for some constant $b\in \mathbb{R}$.
Then for every positive $\varphi \in C^{\infty}(\overline{\Omega})$ we have 
\begin{equation*}
    \int_{\Omega} |\nabla(\e^{\frac{2f}{n}} \varphi)| \ \d m_f + \int_{\partial \Omega}\e^{\frac{2f}{n}} \varphi\ \d m_{f,\partial \Omega} \geq n\,\e^{\frac{2b}{n}} \mathcal{V}^{\frac{1}{n}}\left( \int_{\Omega} \varphi^{\frac{n}{n-1}}\ \d m_f \right)^{\frac{n-1}{n}},
\end{equation*}
where we set
\begin{equation*}
\mathcal{V} := \limsup_{r\rightarrow +\infty}\frac{m_f(B_r(o))}{r^n}
\end{equation*}
for a base point $o\in M$.
\end{thm}

\begin{rem}
The quantity $\mathcal{V}$ does not depend on the choice of the base point.
Also,
in the unweighted case,
the existence of the limit is guaranteed by the Bishop-Gromov comparison,
which is called the \textit{asymptotic volume ratio}.
\end{rem}

\begin{rem}
Under the same setting as in Theorem \ref{thm:Sobolev},
the finiteness of $\mathcal{V}$ follows from \cite[Theorem 2.7]{WY},
which is a Bishop type estimate under $\Ric^1_f\geq 0$.
\end{rem}

\begin{rem}
Theorem \ref{thm:Sobolev} recovers the original one in \cite{B} by considering the case of $f\equiv b$.
\end{rem}

Let us prove Theorem \ref{thm:Sobolev}.
\begin{proof}[Proof of Theorem \ref{thm:Sobolev}]
We prove the desired statement along the line of the argument of the proof of \cite[Theorem 1.1]{B}.
By rescaling,
we may assume 
\begin{equation}\label{eq:Sobolev-1}
    \int_{\Omega} |\nabla(\e^{\frac{2f}{n}} \varphi)|\ \d m_f + \int_{\partial \Omega} \e^{\frac{2f}{n}} \varphi \ \d m_{f, \partial \Omega} = n \int_{\Omega} \varphi^{\frac{n}{n-1}}\  \d m_f.
\end{equation}
Let $u$ be the solution to the following boundary value problem:
\begin{equation}\label{eq:Sobolev-boundary-problem}
    \begin{cases}
        \mbox{div}\left(\e^{-\left(1-\frac{2}{n}\right)f}\varphi\nabla u \right) = \e^{-f}\left(n\varphi^{\frac{n}{n-1}} - |\nabla (\e^{\frac{2f}{n}}\varphi)|\right) & \mbox{ on } \Omega,\\
         u_\nu = 1 & \mbox{ on } \partial \Omega,
     \end{cases}
\end{equation}
where $\nu$ is the outer unit normal vector field on $\partial \Omega$.
Since $\varphi$ is smooth and the right-hand side of the equation is $C^{0,1}$, 
we have $u \in C^{2,\alpha}(\overline{\Omega})$ for any $\alpha \in (0,1)$ by the standard elliptic regularity theory (see e.g., \cite[Theorem 6.30]{GT}).
We set
\begin{equation*}
U := \{ x \in \Omega\, \mid\,  |(\nabla u)_x| < 1\}.
\end{equation*}
Notice that
on $U$, 
\begin{equation}\label{eq:Sobolev-10}
    \e^{\frac{2f}{n}} \varphi \Delta_f u = n \varphi^{\frac{n}{n-1}} - \langle \nabla (\e^{\frac{2f}{n}}\varphi),\nabla u \rangle - |\nabla (\e^{\frac{2f}{n}}\varphi)| \leq n\varphi^{\frac{n}{n-1}}
\end{equation}
by \eqref{eq:Sobolev-boundary-problem} and the Cauchy-Schwarz inequality.

For $r>0$,
let $A_r$ stand for the set of all $x\in U$ such that
\begin{equation*}
r\,u(y)+\frac{1}{2}d(y,\exp_x(r(\nabla u)_x))^2\geq r\,u(x)+\frac{r^2}{2}|(\nabla u)_x|^2
\end{equation*}
for all $y\in \overline{\Omega}$.
We also define a map $\Phi_r:\overline{\Omega}\to M$ by
\begin{equation*}
\Phi_r(x):=\exp_x(r(\nabla u)_x).
\end{equation*}
Furthermore,
let $E_r$ denote the set of all $x\in M$ such that $d(x,y) < r$ for all $y\in \overline{\Omega}$.
Due to \cite[Lemma 2.2]{B},
$E_r$ is contained in $\Phi_r(A_r)$.
Hence,
by the same calculation as in the proof of Theorem \ref{thm:ABP},
and by \eqref{eq:Sobolev-10},
we obtain
\begin{align*}
m_f(E_r)&\leq  \int_{A_r}\,\e^{-f(\Phi_{r}(x))} |\det (d \Phi_{r})_x)|\ \d v_g(x)\\
             &\leq  \int_{A_r} \left( 1 + \frac{r}{n}\,  \e^{\frac{2f}{n}}s_{r, f, u} \,\Delta_f u\right)^n \ \d m_f \leq  \int_{U} \left(1 + r\,\e^{-\frac{2b}{n}} \varphi^{\frac{1}{n-1}}\right)^n \ \d m_f,
\end{align*}
where $s_{r,f,u}:\overline{\Omega}\to \mathbb{R}$ is a function defined by
\begin{equation*}\label{def:h-s}
    s_{r,f,u}(x) :=  \int_0^1 \exp\left\{ -\frac{2}{n} f\left(\exp_x\left( \xi \,r\,(\nabla u)_x \right)\right) \right\} \d\xi.
\end{equation*}
Dividing both sides by $r^{n}$ and letting $r \to +\infty$, 
we obtain
\begin{equation*}
 \mathcal{V}  =\limsup_{r \rightarrow +\infty} \frac{m_{f}(E_r)}{r^{n}}\leq \e^{-2b} \int_U \varphi^{\frac{n}{n-1}}  \ \d m_f \leq \e^{-2b} \int_{\Omega} \varphi^{\frac{n}{n-1}} \ \d m_f.
\end{equation*}
This together with \eqref{eq:Sobolev-1} implies
\begin{align*}
    \int_{\Omega} |\nabla (\e^{\frac{2f}{n}} \varphi)| \ \d m_f+\int_{\partial \Omega} \e^{\frac{2f}{n}} \varphi\ \d m_{f,\partial\Omega}
    &= n\left(\int_{\Omega}  \varphi^{\frac{n}{n-1}} \ \d m_f \right)^{\frac{1}{n}}\left(\int_{\Omega}  \varphi^{\frac{n}{n-1}} \ \d m_f \right)^{\frac{n-1}{n}}  \\
    &\geq n  \e^{\frac{2b}{n}} \mathcal{V}^{\frac{1}{n}}\left(\int_{\Omega}  \varphi^{\frac{n}{n-1}}\ \d m_f \right)^{\frac{n-1}{n}}.
\end{align*}
We complete the proof.
\end{proof}

By applying Theorem \ref{thm:Sobolev} to $\varphi := \e^{-\frac{2f}{n}}$,
we obtain the following isoperimetric inequality (cf. \cite[Corollary 1.3]{B}, and also \cite{APPS}, \cite{BK}, \cite{CM1}, \cite{CM2} in the non-smooth framework):
\begin{cor}\label{cor:isoperimetric}
Let $(M,g,f)$ be an $n$-dimensional complete weighted Riemannian manifold, and let $\Omega$ be a bounded domain in $M$ with smooth boundary. 
We assume $\Ric_f^0 \geq 0$ and $f \geq b$ for some constant $b\in \mathbb{R}$.
Then we have 
\begin{equation*}
    m_{f,\partial \Omega}(\partial \Omega) \geq n\e^{\frac{2b}{n}}\mathcal{V}^{\frac{1}{n}}\left(m_{\left(1 + \frac{2}{n-1}\right)f}(\Omega)\right)^{\frac{n-1}{n}}.
\end{equation*}
\end{cor}

\subsection*{{\rm{Acknowledgements}}}

The authors express their gratitude to Professor Shin-ichi Ohta for valuable comments.
The authors thank the anonymous referee for useful comments.
A portion of this work was written while the first named author was at Department of Mathematics, 
Osaka University.
The first named author was supported by JSPS KAKENHI (25KJ0271) and JST, 
the establishment of university fellowships towards the creation of science technology innovation (JPMJFS2125).
The second named author was supported by JSPS KAKENHI (22H04942, 23K12967).


\begin{thebibliography}{99}
    \bibitem{AGS}  L. Ambrosio, N. Gigli and G. Savar\'{e}, \textit{Metric measure spaces with Riemannian Ricci curvature bounded from below}, Duke Math. J. {\bf 163} (2014), no. 7, 1405--1490.
    \bibitem{APPS}  G. Antonelli, E. Pasqualetto, M. Pozzetta and D. Semola, \textit{Asymptotic isoperimetry on non collapsed spaces with lower Ricci bounds}, Math. Ann. {\bf 389} (2024), no.2, 1677--1730.
    \bibitem{BE}  D. Bakry and M. \'Emery, \textit{Diffusions hypercontractives}, S\'eminaire de probabilit\'es, XIX, 1983/84, 177--206, Lecture Notes in Math., 1123, Springer, Berlin, 1985.
    \bibitem{BK}  Z. M. Balogh and A. Krist\'{a}ly, \textit{Sharp isoperimetric and Sobolev inequalities in spaces with nonnegative Ricci curvature}, Math. Ann. {\bf 385} (2023), no. 3-4, 1747--1773.
    \bibitem{BS}  M. Batista and J. I. Santos, \textit{Manifolds with Density and the First Steklov Eigenvalue}, Potential Anal. {\bf 60} (2024), 1369--1382.
    \bibitem{B}  S. Brendle, \textit{Sobolev inequalities in manifolds with nonnegative curvature}, Comm. Pure Appl. Math. {\bf 76} (2023), no. 9, 2192--2218.
    \bibitem{C} X. Cabr\'{e}, \textit{Nondivergent elliptic equations on manifolds with nonnegative curvature}, Comm. Pure Appl. Math. {\bf 50} (1997), no. 7, 623--665.
    \bibitem{CM1} F. Cavalletti and D. Manini, \textit{Isoperimetric inequality in noncompact $\MCP$ spaces}, Proc. Amer. Math. Soc. {\bf 150} (2022), no. 8, 3537--3548.
    \bibitem{CM2} F. Cavalletti and D. Manini, \textit{Rigidities of Isoperimetric inequality under nonnegative Ricci curvature}, preprint arXiv:2207.03423v2, to appear in J. Eur. Math. Soc..
    \bibitem{CMZ}  X. Cheng, T. Mejia and D. Zhou, \textit{Eigenvalue estimate and compactness for closed $f$-minimal surfaces}, Pacific J. Math. {\bf 271} (2014), no. 2, 347--367.
    \bibitem{CS}  H. I. Choi and R. Schoen, \textit{The space of minimal embeddings of a surface into a three-dimensional manifold of positive Ricci curvature}, Invent. Math. {\bf 81} (1985), no. 3, 387--394.
    \bibitem{CW}  H. I. Choi and A. N. Wang, \textit{A first eigenvalue estimate for minimal hypersurfaces}, J. Differential Geom. {\bf 18} (1983), no. 3, 559--562.
    \bibitem{F} Y. Fujitani, \textit{Analysis of harmonic functions under lower bounds of $N$-weighted Ricci curvature with $\varepsilon$-range}, J. Math. Anal. Appl. {\bf 542} (2025), no. 2, Paper No. 128848, 37 pp.
\bibitem{GT} D. Gilbarg and N. S. Trudinger, \textit{Elliptic partial differential equations of second order}, Springer-Verlag, Berlin, 2001.
    \bibitem{H} B.-X. Han, \textit{$\ABP$ estimate on metric measure spaces via optimal transport}, preprint arXiv:2408.10725.
    \bibitem{HR} Q. Huang and Q. Ruan, \textit{Applications of some elliptic equations in Riemannian manifolds}, J. Math. Anal. Appl. {\bf 409} (2014), no. 1, 189--196.
    \bibitem{J}  F. Johne, \textit{Sobolev inequalities on manifolds with nonnegative Bakry-\'{E}mery Ricci curvature}, preprint arXiv:2103.08496.
    \bibitem{KW}   L. Kennard and W. Wylie, \textit{Positive weighted sectional curvature}, Indiana Univ. Math. J. {\bf 66} (2017), no. 2, 419--462.
    \bibitem{KWY} L. Kennard, W. Wylie and D. Yeroshkin, \textit{The weighted connection and sectional curvature for manifolds with density}, J. Geom. Anal. {\bf 29} (2019), no. 1, 957--1001.
    \bibitem{Kim} S. Kim, \textit{Harnack inequality for nondivergent elliptic operators on Riemannian manifolds}, Pacific J. Math. {\bf 213} (2004), no. 2, 281--293.
    \bibitem{K}  B. Klartag, \textit{Needle decompositions in Riemannian geometry}, Mem. Amer. Math. Soc. {\bf 249} (2017), no. 1180, v + 77 pp.
    \bibitem{KM1}  A. V. Kolesnikov and E. Milman, \textit{Brascamp-Lieb-type inequalities on weighted Riemannian manifolds with boundary}, J. Geom. Anal. {\bf 27} (2017), no. 2, 1680--1702.
    \bibitem{KM2}  A. V. Kolesnikov and E. Milman, \textit{Poincar\'{e} and Brunn-Minkowski inequalities on the boundary of weighted Riemannian manifolds}, Amer. J. Math. {\bf 140} (2018), no. 5, 1147--1185.
    \bibitem{KL}  K. Kuwae and X.-D. Li, \textit{New Laplacian comparison theorem and its applications to diffusion processes on Riemannian manifolds}, Bull. Lond. Math. Soc. {\bf 54} (2022), no. 2, 404--427.
    \bibitem{KLLS}  K. Kuwae, S. Li, X.-D. Li and Y. Sakurai, \textit{Liouville theorem for $V$-harmonic maps under non-negative $(m,V)$-Ricci curvature for non-positive $m$}, Stochastic Process. Appl. {\bf 168} (2024), Paper No. 104270, 25 pp.
    \bibitem{LW}   H. Li and Y. Wei, \textit{$f$-minimal surface and manifold with positive $m$-Bakry-\'{E}mery Ricci curvature}, J. Geom. Anal. {\bf 25} (2015), no. 1, 421--435.
    \bibitem{LX}   J. Li and C. Xia, \textit{An integral formula for affine connections}, J. Geom. Anal. {\bf 27} (2017), no. 3, 2539--2556.
    \bibitem{Li}   P. Li, \textit{Geometric analysis}, Cambridge Stud. Adv. Math., {\bf 134}, Cambridge University Press, Cambridge, 2012.
    \bibitem{L} A. Lichnerowicz, \textit{Vari\'et\'es riemanniennes \`a tenseur C non n\'egatif}, C. R. Acad. Sci. Paris S\'er. A-B {\bf 271} (1970), A650--A653.
    \bibitem{LV} J. Lott and C. Villani, \textit{Ricci curvature for metric-measure spaces via optimal transport}, Ann. of Math. (2) {\bf 169} (2009), no. 3, 903--991.
    \bibitem{LMO} Y. Lu, E. Minguzzi and S. Ohta, \textit{Comparison theorems on weighted Finsler manifolds and spacetimes with $\eps$-range}, Anal. Geom. Metr. Spaces {\bf 10} (2022), no. 1, 1--30.
    \bibitem{MD} L. Ma and S.-H. Du, \textit{Extension of Reilly formula with applications to eigenvalue estimates for drifting Laplacians}, C. R. Math. Acad. Sci. Paris {\bf 348} (2010), no. 21--22, 1203--1206.
    \bibitem{MR} M. Magnabosco and C. Rigoni, \textit{Optimal maps and local-to-global property in negative dimensional spaces with Ricci curvature bounded from below}, Tohoku Math. J. (2) {\bf 75} (2023), no. 4, 483--507.
    \bibitem{MRS} M. Magnabosco, C. Rigoni and G. Sosa, \textit{Convergence of metric measure spaces satisfying the $\CD$ condition for negative values of the dimension parameter}, Nonlinear Anal. {\bf 237} (2023), Paper No. 113366, 48 pp.
     \bibitem{M} E. Milman, \textit{Beyond traditional Curvature-Dimension I: new model spaces for isoperimetric and concentration inequalities in negative dimension}, Trans. Amer. Math. Soc. {\bf 369} (2017), no. 5, 3605--3637.
    \bibitem{MS} A. Mondino and D. Semola, \textit{Lipschitz continuity and Bochner-Eells-Sampson inequality for harmonic maps from $\RCD(K,N)$ spaces to $\CAT(0)$ spaces}, to appear in Amer. J. Math..
    \bibitem{O1}  S. Ohta, \textit{$(K,N)$-convexity and the curvature-dimension condition for negative $N$}, J. Geom. Anal. {\bf 26} (2016), no. 3, 2067--2096.
    \bibitem{O2}  S. Ohta, \textit{Needle decompositions and isoperimetric inequalities in Finsler geometry}, J. Math. Soc. Japan {\bf 70} (2018), no. 2, 651--693.
    \bibitem{O}  S. Oshima, \textit{Stability of curvature-dimension condition for negative dimensions under concentration topology}, J. Geom. Anal. {\bf 33} (2023), no. 12, Paper No. 377, 37 pp.
    \bibitem{S1} K.-T. Sturm, \textit{On the geometry of metric measure spaces. I}, Acta Math. {\bf 196} (2006), no. 1, 65--131.
    \bibitem{S2} K.-T. Sturm, \textit{On the geometry of metric measure spaces. I\hspace{-.1em}I}, Acta Math. {\bf 196} (2006), no. 1, 133--177.
    \bibitem{V} C. Villani, \textit{Optimal Transport: Old and New}, Springer-Verlag, Berlin, 2009.
     \bibitem{WX} Q. Wang and C. Xia, \textit{Sharp bounds for the first non-zero Stekloff eigenvalues}, J. Funct. Anal. {\bf 257} (2009), no. 8, 2635--2644.
    \bibitem{WZ}  Y. Wang and X. Zhang, \textit{An Alexandroff-Bakelman-Pucci estimate on Riemannian manifolds}, Adv. Math. {\bf 232} (2013), 499--512.
    \bibitem{WZ2}  Y. Wang and X. Zhang, \textit{Measure estimates, Harnack inequalities and Ricci lower bound}, Univ. Iagel. Acta Math. (2018), no. 55, 21--51.
    \bibitem{WW}  G. Wei and W. Wylie, \textit{Comparison geometry for the Bakry-Emery Ricci tensor}, J. Differential Geom. {\bf 83} (2009), no. 2, 377--405.
    \bibitem{W1}   W. Wylie, \textit{Sectional curvature for Riemannian manifolds with density}, Geom. Dedicata {\bf 178} (2015), 151--169.
    \bibitem{W2} W. Wylie, \textit{A warped product version of the Cheeger-Gromoll splitting theorem}, Trans. Amer. Math. Soc. {\bf 369} (2017), no. 9, 6661--6681.
    \bibitem{WY} W. Wylie and D. Yeroshkin, \textit{On the geometry of Riemannian manifolds with density}, preprint arXiv:1602.08000.
    \bibitem{X}  C. Xia, \textit{Rigidity of compact manifolds with boundary and nonnegative Ricci curvature}, Proc. Amer. Math. Soc. {\bf 125} (1997), no. 6, 1801--1806.
\end{thebibliography}
\end{document}